\documentclass[11pt]{amsart}
\usepackage{amstext,amssymb,amsmath,amsbsy,dsfont,tikz}

 \usepackage[left=2.65cm,right=2.65cm,top=3.2cm,bottom=3.2cm]{geometry}

\usepackage{amsmath,amssymb,latexsym,dsfont}
\usepackage[small]{caption}
\usepackage{graphicx,color,mathrsfs,tikz}
\usepackage{subfigure,color}
\usepackage{cite}
\usepackage[colorlinks=true,urlcolor=blue,
citecolor=red,linkcolor=blue,linktocpage,pdfpagelabels,
bookmarksnumbered,bookmarksopen]{hyperref}
\usepackage[italian,english]{babel}
\usepackage{units}
\usepackage{enumitem}

\usepackage{hyperref}
\usepackage{cleverref}

\usepackage{tikz}
\usetikzlibrary{intersections}

\usepackage{amscd}
\usepackage{amsfonts}
\usepackage{indentfirst}
\usepackage{verbatim}
\usepackage{amsmath}
\usepackage{amsthm}
\usepackage{enumerate}
\usepackage{graphicx}
\usepackage{color}
\usepackage[OT1]{fontenc}
\usepackage[latin1]{inputenc}
\usepackage[english]{babel}
\usepackage{amssymb,esint}

\newtheorem{theorem}{Theorem}[section]
\newtheorem{lemma}{Lemma}[section]
\newtheorem{proposition}{Proposition}[section]

\newtheorem{remark}{Remark}[section]

\usepackage{tikz}
\usetikzlibrary{intersections}

\numberwithin{equation}{section}

\newcommand{\tr}{^\mathsf{T}}

\newcommand{\vertiii}[1]{{\left\vert\kern-0.25ex\left\vert\kern-0.25ex\left\vert #1
    \right\vert\kern-0.25ex\right\vert\kern-0.25ex\right\vert}}

\newcommand{\dsp}{\displaystyle}

\newcommand{\ess}{\mathrm{ess} \,}

\newcommand{\cH}{{\mathcal H}}

\newcommand{\eps}{\varepsilon}

\newcommand{\mR}{\mathbb{R}}

\newcommand{\cV}{\mathcal V}
\newcommand{\cU}{\mathcal U}
\newcommand{\cW}{\mathcal W}

\newcommand{\hcV}{\hat{\mathcal V}}

\newcommand{\hcW}{\hat{\mathcal W}}

\newcommand{\tcV}{\widetilde{\mathcal V}}

\newcommand{\tcW}{\widetilde{\mathcal W}}

\newcommand{\cB}{\mathcal B}

\newcommand{\B}{\mathrm{B}}

\title[Lyapunov functions]{
Lyapunov functions and  finite time stabilization in optimal time for homogeneous linear and quasilinear  hyperbolic systems}
\author{Jean-Michel Coron}
\address[Jean-Michel Coron]{Sorbonne Universit\'{e}, Universit\'{e} de Paris, CNRS, INRIA,
\newline \indent 	Laboratoire Jacques-Louis Lions, \'{e}quipe Cage, Paris, France.}
\email{coron@ann.jussieu.fr.}

\author{Hoai-Minh Nguyen}
\address[Hoai-Minh Nguyen]{Ecole Polytechnique F\'ed\'erale de Lausanne, EPFL,
\newline \indent CAMA, Station 8,  CH-1015 Lausanne, Switzerland.}
\email{hoai-minh.nguyen@epfl.ch}

\begin{document}

\maketitle

\begin{abstract} Hyperbolic systems in one dimensional space are frequently used
in modeling of many physical systems. In our recent works, we introduced time independent feedbacks leading to the finite stabilization for the optimal time of homogeneous linear and quasilinear  hyperbolic systems. In this work, we present Lyapunov's functions for these feedbacks and use estimates for Lyapunov's functions to rediscover the finite stabilization results. 
\end{abstract}

\medskip 
\noindent{\bf Keywords}: hyperbolic systems, boundary controls, Lyapunov functions, backstepping, finite time stabilization,  optimal time.

\section{Introduction}

Hyperbolic systems in one dimensional space are frequently used
in modeling of many systems such as traffic flow \cite{SFA08}, heat exchangers \cite{XS02},  and fluids in open channels \cite{GL03, dHPCAB03, GLS04, DP08}, transmission lines \cite{CFM11}, phase transition \cite{Goatin06}.  In our recent works \cite{CoronNg19, CoronNg20}, we introduced time independent feedbacks leading to the finite stabilization for the optimal time of
homogeneous linear and quasilinear  hyperbolic systems.
In this work, we present Lyapunov's functions for these feedbacks and use estimates for Lyapunov's functions to rediscover the finite stabilization results.  More precisely,  we are concerned about the following homogeneous,  quasilinear,  hyperbolic system in one dimensional space
\begin{equation}\label{Sys}
\partial_t w (t, x) =  \Sigma \big(x, w(t, x) \big) \partial_x w (t, x)  \mbox{ for } (t, x)  \in [0, + \infty) \times (0, 1).
\end{equation}
Here $w = (w_1, \cdots, w_n)\tr: [0, + \infty) \times (0, 1) \to \mR^n$, $\Sigma(\cdot, \cdot)$ is an   $(n \times n)$ real matrix-valued function defined in $[0, 1] \times \mR^n$. We assume that  $\Sigma (\cdot, \cdot)$ has  $m \ge 1$  distinct positive eigenvalues,  and $k = n - m \ge 1$  distinct negative eigenvalues. We also assume that, maybe after a change of variables,   $\Sigma(x, y)$ for $x \in [0, 1]$ and   $y \in \mR^n$ is of the form
\begin{equation}\label{form-A}
\Sigma(x, y) = \mbox{diag} \Big(- \lambda_1(x, y), \cdots, - \lambda_{k}(x, y),  \lambda_{k+1}(x, y), \cdots,  \lambda_{k+m}(x, y) \Big),
\end{equation}
where
\begin{equation}\label{relation-lambda}
-\lambda_{1}(x, y) < \cdots <  - \lambda_{k} (x, y)< 0 < \lambda_{k+1}(x, y) < \cdots \lambda_{k+m}(x, y).
\end{equation}
Throughout the paper, we assume
\begin{equation}\label{cond-lambda}
\mbox{$\lambda_i$ and $\partial_y \lambda_i$  are of class $C^1$ with respect to $x$ and $y$    for $1 \le i \le n = k + m$.}
\end{equation}
Denote
$$
\mbox{$w_- = (w_1, \cdots, w_k)\tr $ and $w_+ = (w_{k+1}, \cdots, w_{k+m})\tr$.}
$$

The following types of boundary conditions and controls are considered.
The boundary condition at $x = 0$ is given by
\begin{equation}\label{bdry-w-0}
w_-(t, 0)  = \B \big(w_+(t,  0) \big) \mbox{ for } t \ge 0,
\end{equation}
for some
$$
\mbox{$\B \in \big( C^2(\mR^m) \big)^k$ \mbox{with} $\B(0) = 0$,}
$$
and the boundary control at $x = 1$ is
\begin{equation}\label{bdry-w-1}
w_+(t, 1) = (W_{k+1}, \cdots, W_{k+m})\tr (t) \mbox{ for } t \ge 0,
\end{equation}
where $W_{k +1}, \dots, W_{k + m}$ are controls.

Set
\begin{equation}\label{def-tau}
\tau_i =  \int_0^1  \frac{1}{\lambda_i(x, 0)} \, dx \quad  \mbox{ for } 1 \le i \le n.
\end{equation}

The exact controllability, the null-controllability, and the boundary stabilization of  hyperbolic systems in one dimension  have been widely investigated in the literature for almost half a century, see e.g. \cite{BC16} and the references therein.  Concerning the exact controllability and  the null-controllability related to \eqref{bdry-w-0} and \eqref{bdry-w-1}, the pioneer works date back to Jeffrey Rauch and Michael Taylor \cite{RT74} and David Russell \cite{Russell78} for the linear inhomogeneous system. In the  quasilinear case with $m \ge k$,  the null controllability was established for $m \ge k$ by Tatsien Li in  \cite[Theorem 3.2]{Li00} (see also \cite{LiRao02}). These results hold for the time  $\tau_k + \tau_{k+1}$.

Concerning the stabilisation of \eqref{Sys}, many works are concerned about the boundary conditions of the following specific form
\begin{equation}\label{bdry-S}
\left(\begin{array}{c}
    w_- (t, 0) \\
    w_+ (t, 1)
  \end{array}\right)
= G
 \left( \begin{array}{c}
    w_+(t, 1) \\
    w_- (t, 0)
  \end{array} \right),
\end{equation}
where $G: \mR^n \to \mR^n$ is a suitable smooth vector field. Three approaches have been proposed to deal with \eqref{bdry-S}.  The first one is based on the characteristic method.  This method was  investigated  in the framework of $C^1$-norm \cite{GL84, Li94}. The second one is based on Lyapunov functions \cite{BCN99, LS02, BCN07, CBN08,CBA08, DBC12, BC15}. 
The third one is via the delay equations and was investigated  in the framework of  $W^{2,p}$-norm with $p \ge 1$ \cite{CoronNg15}. 
Surprisingly,  the stability criterion in the nonlinear setting depends on the norm considered \cite{CoronNg15}. Required assumptions impose some restrictions on the magnitude of the coupling coefficients when dealing with inhomogeneous systems.

Another way to
stabilise \eqref{Sys} is to use  the backstepping approach. This was first proposed by Jean-Michel Coron et al.  \cite{CVKB13} for $2\times 2$ inhomogeneous  system $(m=k=1)$.  Later this approach has been extended and now can be applied for general pairs $(m,k)$ in  the linear case \cite{MVK13, HMVK16, AM16, CHO17, CoronNg19, CoronNg19-2}.  In \cite{CVKB13}, the authors obtained feedbacks leading to the finite stabilization in time $\tau_1+\tau_2$  with $m = k = 1$.   In \cite{HMVK16}, the authors considered the case where $\Sigma$ is constant and   obtained feedback laws for the null-controllability at the time $\tau_k + \sum_{l = 1}^m \tau_{k + l}$.  Later \cite{AM16, CHO17}, feedbacks leading to the finite stabilization in time $ \tau_k + \tau_{k+1}$ were derived.

Set, as in \cite{CoronNg19, CoronNg20}
\begin{equation}\label{def-Top}
T_{opt} := \left\{ \begin{array}{cl}  \dsp \max \Big\{ \tau_1 + \tau_{m+1}, \dots, \tau_k + \tau_{m+k}, \tau_{k+1} \Big\} & \mbox{ if } m \ge k, \\[6pt]
\dsp \max \Big\{ \tau_{k+1-m} + \tau_{k+1},  \tau_{k+2-m} + \tau_{k+2},  \dots, \tau_{k} + \tau_{k+m} \Big\} &  \mbox{ if } m < k.
\end{array} \right.
\end{equation}
Define
\begin{equation}
\cB: = \Big\{B \in \mR^{k \times m}; \mbox{ such that  \eqref{cond-B-1} holds for  $1 \le i \le \min\{m-1, k\}$} \Big\},
\end{equation}
where
\begin{multline}\label{cond-B-1}
\mbox{ the $i \times i$  matrix formed from the last $i$ columns and the last $i$ rows of $B$  is invertible.}
\end{multline}
Using the backstepping approach, we  established  the null-controllability for the  linear inhomogeneous systems for the optimal time $T_{opt}$ under the condition $B := \nabla \B (0) \in \cB$ \cite{CoronNg19, CoronNg19-2} (see also \cite{CoronNg20} for the non-linear, homogeneous case).  This condition is very natural to obtain the null-controllability at $T_{opt}$ which roughly speaking  allows to  use the $l$ controls  $W_{k+ m - l + 1}, \cdots, W_{k+m}$  to control the $l$ directions $w_{k-l +1}, \cdots, w_{k}$ for $1 \le l \le \min\{k, m\}$ (the possibility to implement $l$ controls corresponding to the fastest  positive speeds to control  $l$ components corresponding to the lowest negative speeds \footnote{The $i$  direction ($1 \le i \le n$) is called positive (resp. negative) if $\lambda_i$ is positive (resp. negative).}). The optimality of $T_{opt}$ was given in \cite{CoronNg19}  (see also \cite{Weck}). Related exact controllability results can be also found in  \cite{LongHu, CoronNg19, CoronNg19-2, HO19}.   It is easy to see  that $\cB$ is an open subset of  the set of (real) $k\times m$ matrices  and the Hausdorff dimension of its complement is $\min\{k, m-1 \}$.

We previously obtained time independent feedbacks leading finite stabilization  for the optimal time $T_{opt}$ of the system \eqref{Sys}, \eqref{bdry-w-0}, and \eqref{bdry-w-1} when $B   \in {\mathcal B}$  in the linear case \cite{CoronNg19},  and in the nonlinear case \cite{CoronNg20}.
In this paper, we introduce Lyapunov functions for these feedbacks. As a consequence of our estimate on the decay rate of solutions via the Lyapunov functions (\Cref{thm1} and \Cref{thm2}), we are able to rediscover the finite stabilization results in the optimal time  \cite{CoronNg19, CoronNg20}.

To keep the notations simple in the introduction, from now on, we only discuss the linear setting, i.e., $\Sigma(x, y) = \Sigma(x)$ (so $\lambda_i(x, y) = \lambda_i(x)$) and $\B (\cdot) = B \cdot $ (recall that $B = \nabla \B (0)$). The nonlinear setting will be discussed in \Cref{sect-NL}.
The boundary condition at $x=0$ becomes
\begin{equation}\label{bdry-w-0-LN}
w_-(t, 0)  = B w_+(t,  0) \mbox{ for } t \ge 0.
\end{equation}
 We first introduce/recall some notations.
Extend $\lambda_i$ in $\mR$  with $1 \le i \le k+m$ by $\lambda_i(0)$ for $x< 0$ and $\lambda_i(1)$ for $x >1$.  For $(s, \xi) \in [0, T] \times [0, 1]$, define $x_i(t, s, \xi)$ for $t \in \mR$ by
\begin{equation}\label{def-xi-1}
\frac{d}{d t} x_i(t, s, \xi) = \lambda_i \big(x_i(t, s, \xi) \big) \mbox{ and }  x_i(s, s, \xi) = \xi \mbox{ if } 1 \le i \le k,
\end{equation}
and
\begin{equation}\label{def-xi-2}
\frac{d}{d t} x_i(t, s, \xi) = - \lambda_i \big(x_i(t, s, \xi) \big) \mbox{ and }  x_i(s, s, \xi) = \xi \mbox{ if } k+1 \le i \le k+ m
\end{equation}
(see \Cref{fig1}).

\begin{figure}
\centering
\begin{tikzpicture}[scale=3]

\newcommand\z{0.0}

\draw[->] (0+\z,0) -- (1.7+\z,0);
\draw[->] (0+\z,0) -- (0+\z,2.4);
\draw[] ({1.5*(1+\z},0) -- ({1.5*(1+\z)},2.4);
\draw (1.5+\z, 0) node[below]{$1$};
\draw (\z, 2.4) node[left]{$t$};
\draw (\z, 0) node[left, below]{$0$};

\draw[blue] ({1.5*(1+\z)}, 0) -- (\z, 0.5);
\draw[orange] ({1.5*\z}, 0.5) node[left, orange]{$\tau_6$};
\draw[orange] ({1.5*(\z + 0.4)}, 0.3) node[above, orange]{\small $x_6$};

\draw[dashed, orange] ({1.5*\z}, 0.5) -- ({1.5*(\z + 0.625)}, 0);
\draw ({1.5*(\z + 0.625)}, -0.02) node[below, orange]{\small $x_5(0, \tau_6,  0)$};

\draw[dashed, orange] ({1.5*\z}, 0.5) -- ({1.5*(\z + 0.3125)}, 0);
\draw ({1.5*(\z + 0.3125)}, -0.26) node[below, orange]{\small $x_4(0, \tau_6,  0)$};

\draw[red] ({1.5*\z}, 0.5) -- ({1.5*(\z + 1)}, 2);
\draw ({1.5*(\z + 1)}, 2) node[right, orange]{$\tau_3 + \tau_6$};
\draw[orange] ({1.5*(\z + 0.25)}, 0.8) node[below, orange]{\small $x_3$};


\draw[blue] ({1.5*(1+\z)}, 0) -- ({1.5*\z}, 0.8);
\draw ({1.5*\z}, 0.8) node[left, violet]{$\tau_5$};
\draw[orange] ({1.5*(\z + 0.4)}, 0.5) node[above, violet]{\small $x_5$};
\draw[dashed, violet] ({1.5*\z}, 0.8) -- ({1.5*(\z + 0.5)}, 0);
\draw ({1.5*(\z + 0.5)}, -0.14) node[below, violet]{\small $x_4(0, \tau_5,  0)$};
\draw[red] ({1.5*\z}, 0.8) -- ({1.5*(\z + 1)}, 1.8);
\draw ({1.5*(\z + 1)}, 1.8) node[right, violet]{$\tau_{2} + \tau_5$};
\draw[violet] ({1.5*(\z + 0.16)}, 1) node[above, violet]{\small $x_2$};

\draw[blue] ({1.5*(1+\z)}, 0) -- (\z, 1.6);
\draw ({1.5*\z}, 1.6) node[left]{$\tau_4$};
\draw ({1.5*(\z + 0.8)}, 0.5) node[above]{\small $x_4$};

\draw[red] ({1.5*\z}, 1.6) -- ({1.5*(\z + 1)}, 2.25);
\draw ({1.5*(\z + 1)}, 2.25) node[right]{$\tau_1 + \tau_4$};
\draw ({1.5*(\z + 0.5)}, 2) node[above]{\small $x_1$};

\draw[] (\z + 0.85, -0.6) node[left]{$a)$};


\newcommand\zz{2.5}

\draw[->] (0+\zz,0) -- (1.7+\zz,0);
\draw[->] (0+\zz,0) -- (0+\zz,2.4);
\draw[] (1.5+\zz,0) -- (1.5+\zz,2.4);

\draw[blue] (1.5+\zz, 0) -- (\zz, 0.75);
\draw[] (\zz + 0.3, 0.7) node[above]{\small $x_j(\cdot, 0, 1)$};

\draw[orange] (1.5+\zz, 0) -- (\zz, 1.6);

\draw[] (\zz + 1.1, 0.7) node[above]{\small $x_i(\cdot, 1, 0)$};

\draw (1.5+\zz, 0) node[below]{$1$};

\draw (\zz, 2.4) node[left]{$t$};

\draw (\zz, 0) node[left, below]{$0$};

\draw[blue, dashed] (0.4*15/7.5+\zz, 0) -- (\zz, 0.4);
\draw (0.4*15/7.5+\zz, 0) node[below]{$x$};

\draw[orange, dashed] (0.375+\zz, 0) -- (\zz, 0.4);
\draw (\zz + 0.4, 0) node[below, orange]{$a_{i, j}(x)$};

\draw[] (\zz, 0.4) node[left]{$\tau(j, x)$};
\draw[] (\zz + 0.55, 0.20) node[above]{\small $x_j(\cdot,  0, x)$};

\draw[] (\zz + 0.85, -0.6) node[left]{$b)$};

\end{tikzpicture}
\caption{a) $k = m = 3$,  $\Sigma$ is constant,  $x_1 = x_1(\cdot, \tau_4, 0)$,  $x_2 = x_2 (\cdot, \tau_5, 0)$,  $x_3 = x_3 (\cdot, \tau_4, 0)$,  $x_4 = x_4(\cdot, 0, 1)$, $x_5 = x_5(\cdot, 0, 1)$, and $x_6 = x_6(\cdot, 0, 1)$. b) $k+1 \le i  < j \le k+m$, and $\Sigma$ is constant.}\label{fig1}
\end{figure}
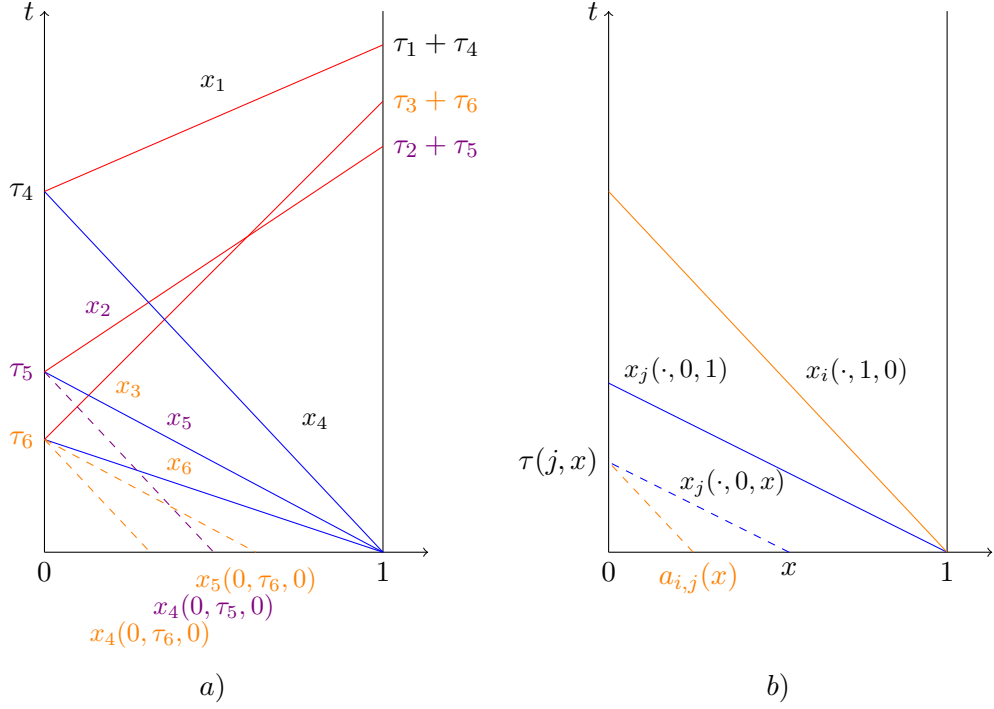
For $x \in [0, 1]$, and $k+1 \le j \le k + m$,  let $\tau(j, x) \in \mR_+$ be such that
$$
x_{j} \big(\tau(j, x), 0, x \big) = 0,
$$
and set, $k+1 \le i < j \le k + m$,
\begin{equation}\label{thm1-aij}
a_{i, j}(x) = x_{i} \big(0, \tau (j, x), 0 \big)
\end{equation}
(see \Cref{fig1}-b)). It is clear that $\tau(j, 1) = \tau_j$ for $k+1 \le j \le k+m$.

We now recall the feedback in \cite{CoronNg19}. We first consider the case $m \ge k$.  Using \eqref{cond-B-1} with $i = 1$, one can derive that $w_{k}(t, 0) = 0$ if and only if
\begin{equation}\label{kd-1}
w_{m+k}(t, 0) = M_k (w_{k +1 }, \cdots, w_{m + k - 1})\tr (t, 0),
\end{equation}
for some  constant  matrix  $M_k$ of size $1 \times (m  -1)$. Using \eqref{cond-B-1} with $i = 2$, one can derive that $w_{k}(t, 0) = w_{k-1}(t, 0)  = 0$ if and only if  \eqref{kd-1} and
\begin{equation}\label{kd-2}
w_{m+k-1}(t, 0) = M_{k-1} (w_{k +1 }, \cdots, w_{m + k - 2})\tr (t, 0)
\end{equation}
hold for some  constant  matrix  $M_{k-1}$ of size $1 \times (m  -2)$ by the Gaussian elimination method, etc. Finally, using \eqref{cond-B-1} with $i = k$, one can derive that $w_{k}(t, 0) = w_{k-1}(t, 0)  \dots = w_1(t, 0)= 0$ if and only if  \eqref{kd-1}, \eqref{kd-2}, \dots, and
\begin{equation}\label{kd-3}
w_{m+1}(t, 0) = M_{1} (w_{k +1 }, \cdots, w_{m})\tr (t, 0)
\end{equation}
hold  for some  constant  matrix  $M_{1}$ of size $1 \times (m  - k)$ by applying \eqref{cond-B-1} with $i= k$ and using the Gaussian elimination method when $m>k$. When $m=k$, similar fact holds with $M_1 = 0$.

The feedback is then given as follows:
\begin{equation}\label{bdry-1}
w_{m + k}(t, 1) = M_{k} \Big(w_{k+ 1}\big(t, x_{k+1} (0, \tau_{m + k},   0) \big), \dots, w_{k+ m-1}\big(t, x_{k+ m - 1} (0, \tau_{m + k },   0) \big)\Big),
\end{equation}
\begin{equation}\label{bdry-2}
w_{m + k - 1}(t, 1) = M_{k-1} \Big(w_{k+ 1}\big(t, x_{k+1} (0,  \tau_{m + k - 1},  0) \big), \dots, w_{k+ m-2}\big(t, x_{k+ m - 2} (0, \tau_{m + k -1},  0) \big) \Big),
\end{equation}
\dots
\begin{equation}\label{bdry-3}
w_{m+ 1}(t, 1) = M_{1} \Big(w_{k+ 1} \big(t, x_{k+1} (0,  \tau_{m + 1},  0) \big), \dots, w_{m}\big(t, x_{m +1} (0, \tau_{m +1},   0) \big)\Big),
\end{equation}
and
\begin{equation}\label{bdry-4}
w_{j}(t, 1)  = 0 \quad  \mbox{ for } k+1 \le j \le m.
\end{equation}
(see \Cref{fig1}-a)). \footnote{In \cite{CoronNg19}, we use $x_{i} (-\tau_j, 0, 0)$ with $k+1 \le i < j \le k+m$ in the feedback above. Nevertheless, $x_{i} (-\tau_j, 0, 0) = x_{i} (0, \tau_j, 0)$. }

We next deal with the case $m<k$. The construction in this case is based on the construction given in the case $m = k$.  The feedback is then given by
\begin{equation}\label{bdry-1-*}
w_{k+m}(t, 1) = M_{k} \Big(w_{k+ 1}\big(t, x_{k+1} (0, \tau_{k+m},   0) \big), \dots, w_{k+ m-1}\big(t, x_{k+m - 1} (0, \tau_{k+m},   0) \big)\Big),
\end{equation}
\begin{equation}\label{bdry-2-*}
w_{k+m - 1}(t, 1) = M_{k-1} \Big(w_{k+ 1}\big(t, x_{k+1} (0,  \tau_{k+m - 1},  0) \big), \dots, w_{k+m  -2}\big(t, x_{k + m - 2} (0, \tau_{k + m  -1},  0) \big) \Big),
\end{equation}
\dots
\begin{equation}\label{bdry-2-***}
w_{k+2}(t, 1) = M_{k+2 -m} \Big(w_{k+ 1}\big(t, x_{k+1} (0,  \tau_{k+m - 1},  0) \big) \Big),
\end{equation}
\begin{equation}\label{bdry-3-*}
w_{k+ 1}(t, 1) =  M_{k+1 -m},
\end{equation}
with the convention $M_{k+1 -m} =0$.

\begin{remark}\rm
The well-posedness of \eqref{Sys} with $\Sigma (x, y) = \Sigma(x)$, \eqref{bdry-w-0}, with the feedback given above for $w_0 \in \big[ L^\infty(0, 1) \big]^n$ is given by \cite[Lemma 3.2]{CoronNg19}. More precisely, for $w_0 \in \big[ L^\infty(0, 1) \big]^n$ and $T>0$, there exists a unique broad solution $w \in \big[L^\infty \big((0, T) \times [0, 1] \big) \big]^n \cap \big[C\big([0, T]); L^2(0, 1)\big)\big]^n \cap \big[C\big([0, 1]); L^2(0, T)\big)\big]^n$. The broad solutions are defined in \cite[Definition 3.1]{CoronNg19}. The proof is based on a fixed point argument using the norm
$$
\|w \| =  \sup_{1 \le i \le n} \mathop{\ess \sup}_{(\tau, \xi) \in (0, T) \times (0, 1)} e^{-L_1 \tau - L_2 \xi} |w_i(\tau, \xi)|, 
$$
where $L_1, L_2$ are two large positive numbers with  $L_1$ much larger than $L_2$. 
\end{remark}

Concerning these feedbacks, we have

\begin{theorem}\label{thm1} Let $m, \, k \ge 1$, and $w_0 \in \big[L^\infty(0, 1) \big]^n$. There exists a constant $C \ge 1$, depending only on $B$ and $\Sigma$, such that for all $q \ge 1$ and $\Lambda \ge 1$,  it holds
\begin{equation}\label{thm1-cl1}
\| w(t, \cdot) \|_{L^q(0, 1)} \le C e^{\Lambda \big( T_{opt} - t\big)} \| w(t = 0, \cdot) \|_{L^q(0, 1)} \mbox{ for } t \ge 0.
\end{equation}
As a consequence, we have
\begin{equation}\label{thm1-cl2}
\| w(t, \cdot) \|_{L^\infty(0, 1)} \le C e^{\Lambda \big( T_{opt} - t\big)} \| w(t = 0, \cdot) \|_{L^\infty(0, 1)} \mbox{ for } t \ge 0.
\end{equation}

\end{theorem}

As a consequence of \Cref{thm1}, the finite stabilization   in the optimal time $T_{opt}$ is achieved  by taking $\Lambda \to + \infty$ since $C$ is independent of $\Lambda$. The spirit of deriving appropriate information for $L^\infty$-norm from the one associated to  $L^q$-norm was also considered in \cite{BC15}.  The proof of \Cref{thm1} is based on considering the following Lyapunov function. Let $q \ge 1$ and let $\cV: [L^q(0, 1)]^n  \to \mR$ be defined by, with $\ell = \max\{m, k\}$,
\begin{multline}\label{def-V}
\cV(v) =  \sum_{i=1}^{\ell} \int_0^1  p_i(x)|v_i(x)|^q \, dx \\[6pt]
+ \mathop{\sum_{i}}_{\ell + 1 \le m + i \le k+m} \int_0^1   p_{m +i}(x)  \Big| v_{m+ i}(x) - M_{i} \Big(v_{k+ 1} \big(a_{k+1, m + i}(x) \big), \dots, v_{m + i -1}\big( a_{m +i - 1,  m +i}(x) \big)\Big) \Big|^q  \, dx,
\end{multline}
where
\begin{equation}\label{thm1-pA1}
p_i(x) = \lambda_i^{-1}(x) e^{ - q\Lambda \int_0^x \lambda_i^{-1}(s) \, ds + q \Lambda \int_0^1 \lambda_i^{-1}(s) \, ds} \quad \mbox{ for } 1 \le i \le k,
\end{equation}
\begin{equation}\label{thm1-pA2}
p_i(x) = \Gamma^q \lambda_i^{-1}(x) e^{q \Lambda \int_0^x \lambda_i^{-1}(s) \, ds} \quad \mbox{ for } k + 1 \le i \le \ell,
\end{equation}
\begin{equation}\label{thm1-pA3}
p_{m +i} (x) = \Gamma^q \lambda_{m+i}^{-1} (x) e^{q \Lambda \int_0^x \lambda_{m+i}^{-1}(s) \, ds + q \Lambda \int_0^1 \lambda_{i}^{-1}(s) \, ds} \quad \mbox{ for } \ell + 1 \le m + i \le m + k.
\end{equation}
for some large positive constant $\Gamma \ge 1$ depending only on $\Sigma$ and $B$ (it is independent of $\Lambda$ and $q$).

\begin{remark}  \rm Our Lyapunov functions are explicit. This is useful to study the robustness of our feedback laws with respect to disturbances. The use of Lyapunov functions is a classical tool to study the robustness of feedback laws for control system in finite dimension (see, for example, \cite[Sections 4.6, 4.7, 5.5.2, 11.7]{2009-Malisoff-Mazenc-book}. For 1-D hyperbolic systems Lyapunov functions are in particular used for the study of a classical robustness property called the Input-to-State Stability (ISS); see, for example, \cite{2012-Prieur-Mazenc-SCL,2019-Hayat-preprint, 2020-Ferrante-Prieur-2020, 2020-Weldegiyorgis-Banda-preprint}.
\end{remark}

\medskip

The paper is organized as follows. \Cref{sect-LN}  is devoted to the  proof of \Cref{thm1}. The nonlinear setting is considered in \Cref{sect-NL}. The main result there is \Cref{thm2}, which is a variant of \Cref{thm1}.
In the appendix, we will establish a lemma, which is used in the proof of \Cref{thm1} and \Cref{thm2}.

\section{Analysis for the linear setting - Proof of \Cref{thm1}} \label{sect-LN}

This section containing two subsections is devoted to the proof of \Cref{thm1}. The first one is on the case $m \ge k$ and the second one is on the case $m < k$.

\subsection{Proof of \Cref{thm1} for $m \ge k$} \label{sect-LN1} One can check that $a_{i, j}$ is of class $C^1$ since $\Lambda$ is of class $C^1$ (see, for example, \cite[Chapter V]{1964-Hartman-book}).
We claim that, for $k+1 \le i < j \le k+m$ and for $x \in [0, 1]$,
\begin{equation}\label{thm1-claim1}
a_{i, j}'(x) =  \lambda_{i} \big(a_{i, j}(x) \big)/ \lambda_{j}(x).
\end{equation}
Indeed, by the characteristic method and the definition of $a_{i, j}$ and $\tau(j, \cdot)$ (see  also \Cref{fig1}-b)), we have
\begin{equation*}
a_{i, j} \big( x_j(t, 0, x) \big)  =   x_i \big(t, \tau(j, x), 0 \big) \mbox{ for } 0 \le t \le \tau(j, x).
\end{equation*}
Taking the derivative with respect to $t$ gives
\begin{equation*}
a_{i, j}'  \big( x_j(t, 0, x) \big)  \partial_t x_j(t, 0, x) =   \partial_t x_i \big(t, \tau(j, x), 0 \big).
\end{equation*}
This implies, by the definition of $x_i$ and $x_j$,
\begin{equation*}
a_{i, j}'  \big( x_j(t, 0, x) \big)  \lambda_j \big(x_j(t, 0, x) \big)  =  \lambda_i \big(x_i(t, \tau(j, x), 0) \big).
\end{equation*}
Considering $t =0$, we obtain \eqref{thm1-claim1}.

As a consequence of \eqref{thm1-claim1}, we have \begin{equation}\label{thm1-du}
\partial_x \Big(w_{i} \big(t, a_{i, j}(x) \big)\Big) =  \frac{\lambda_{i} \big(a_{i, j}(x) \big)}{\lambda_{j}(x)} \partial_x w_{i} \big(t, a_{i, j}(x) \big).
\end{equation}
Identity \eqref{thm1-du}  is one of the key ingredients in deriving properties for   $\frac{d}{dt} \cV \big(w(t, \cdot) \big)$, which  will be done next.

In what follows, we assume that $w$ is smooth. The general case will follow by a standard approximation argument. Set
\begin{multline}\label{thm1-S}
S_{m+i}(t, x) = \lambda_{m+i}(x) \partial_x w_{m+ i}(t, x)
- M_{i} \Big(\lambda_{k+1}\big(a_{k+1, m+i}(x) \big) \partial_x w_{k+ 1} \big(t, a_{k+1, m+i}(x) \big), \\[6pt]  \dots, \lambda_{m+i-1}\big(a_{m+i - 1, m+i}(x) \big) \partial_x w_{m+ i -1}\big(t, a_{m+i - 1, m+i}(x) \big)\Big),
\end{multline}
and
\begin{equation}\label{thm1-T}
T_{m+i}(t, x) = w_{m+ i}(t, x) - M_{i} \Big(w_{k+ 1} \big(t, a_{k+1, m+i}(x) \big), \dots, w_{m+ i -1}\big(t, a_{m+i - 1, m+i}(x) \big)\Big).
\end{equation}
Since $M_i$ is constant, it follows from the definition of $\cV(v)$ and \eqref{Sys} that, for $t \ge 0$,
\begin{equation}\label{thm1-dV}
\frac{d}{dt} \cV(w(t, \cdot))  = \cU_1 (t) + \cU_2 (t),
\end{equation}
where
\begin{equation}\label{thm1-w1}
\cU_1 (t)  =  - \sum_{i=1}^k  \int_0^1  p_i(x) \lambda_i(x)  \partial_x |w_i(t, x)|^q \, dx  + \sum_{i=k+1}^m\int_0^1  p_i(x) \lambda_i(x)  \partial_x |w_i(t, x)|^q \, dx,
\end{equation}
and
\begin{equation}\label{thm1-w2}
\cU_2 (t) =   \sum_{i =1}^{k} \int_0^1 q  p_{m +i}(x) S_{m+i}(t, x) |T_{m+i }(t, x)|^{q-2} T_{m+i }(t, x) \, dx.
\end{equation}

We next consider $\cU_1$. An integration by parts yields
\begin{multline}\label{thm1-w1-p1-1}
\cU_1 (t) =   \sum_{i=1}^k \int_0^1  (\lambda_i p_i)'(x)  |w_i(t, x)|^q \, dx  -  \sum_{i=k+1}^m \int_0^1  (\lambda_i p_i)'(x)  |w_i(t, x)|^q \, dx \\[6pt]
- \sum_{i=1}^k \lambda_i (x) p_i(x)  |w_i(t, x)|^q \Big|_0^1 + \sum_{i=k+1}^m \lambda_i (x) p_i(x)  |w_i(t, x)|^q \Big|_0^1.
\end{multline}
Using the feedback \eqref{bdry-4} and the boundary condition \eqref{bdry-w-0}, we obtain
\begin{multline}\label{thm1-w1-p1}
\cU_1 (t) =  \sum_{i=1}^k  \int_0^1  (\lambda_i p_i)'(x)   |w_i(t, x)|^q \, dx  - \sum_{i=k+1}^m  \int_0^1  (\lambda_i p_i)'(x) |w_i(t, x)|^q \, dx \\[6pt]
- \sum_{i=1}^k \lambda_i (1) p_i(1)  |w_i(t, 1)|^q +  \sum_{i=1}^k \lambda_i (0) p_i(0)  |(Bw_+)_i(t, 0)|^q - \sum_{i=k+1}^m \lambda_i (0) p_i(0)  |w_i(t, 0)|^q .
\end{multline}

We next deal with $\cU_2$. Using \eqref{thm1-du}, we derive from the definition of $S_{m+i}$ that
\begin{multline}\label{thm1-S-new}
S_{m+i}(t, x) = \lambda_{m+i}(x) \partial_x w_{m+ i}(t, x)
- \lambda_{m+i} (x)  M_{i} \Big(\partial_x \big( w_{k+ 1} \big(t, a_{k+1, m+i}(x) \big) \big), \\[6pt]  \dots,  \partial_x \big( w_{m+ i -1}\big(t, a_{m+i - 1, m+i}(x) \big) \big)\Big),
\end{multline}
which yields, since $M_i$ is constant,
\begin{equation}\label{thm1-ST}
S_{m+i}(t, x) = \lambda_{m+i}(x) \partial_x T_{m+i}(t, x).
\end{equation}
Combining  \eqref{thm1-w2} and \eqref{thm1-ST},  and integrating  by parts yield
\begin{equation}\label{thm1-w2-p1-1}
\cU_2 (t) =  -  \sum_{i =1}^{k}  \int_0^1 (\lambda_{m+i} p_{m+i})'(x) |T_{m+i}(t, x)|^q + \sum_{i =1}^{k} \lambda_{m+i}(x) p_{m+i} (x) |T_{m+i}(t, x)|^q \Big|_0^1.
\end{equation}
By the feedback laws \eqref{bdry-1}-\eqref{bdry-3}, the boundary term in the RHS of \eqref{thm1-w2-p1-1} is
$$
-  \sum_{i = 1}^{k} \lambda_{m+i} (0) p_{m +i} (0) \Big| w_{m+ i}(t, 0) - M_{i} \Big(w_{k+ 1} (t, 0), \dots, w_{m+ i -1} (t,  0)\Big) \Big|^q .
$$
One then has
\begin{multline}\label{thm1-w2-p1}
\cU_2 (t) =  - \sum_{i =1}^{k}  \int_0^1  (\lambda_{m+i} p_{m+i} )'(x) |T_{m+i}(t, x)|^q  \\[6pt]
-  \sum_{i = 1}^{k} \lambda_{m+i} (0) p_{m +i} (0) \Big| w_{m+ i}(t, 0) - M_{i} \Big(w_{k+ 1} (t, 0), \dots, w_{m+ i -1} (t,  0)\Big) \Big|^q.
\end{multline}

From \eqref{thm1-w1-p1} and \eqref{thm1-w2-p1}, we obtain
\begin{equation}\label{thm1-wW}
\cU_1 (t) + \cU_2 (t) = \cW_1(t) + \cW_2 (t),
\end{equation}
where
 \begin{multline}\label{thm1-W1}
 \cW_1(t) = - \sum_{i = 1}^k  \lambda_i (1) p_i (1) |w_i(t, 1)|^q  +  \sum_{i = 1}^k  \lambda_i (0) p_i (0) | (B w_+)_i(t, 0)|^q \\[6pt]  - \sum_{i=k+1}^m \lambda_i (0) p_i(0)  |w_i(t, 0)|^q
  -   \sum_{i = 1}^{k} \lambda_{m+i} (0) p_{m +i} (0) \Big| w_{m+ i}(t, 0) - M_{i} \Big(w_{k+ 1} (t, 0), \dots, w_{m+ i -1} (t,  0)\Big) \Big|^q,
\end{multline}
and
\begin{multline}\label{thm1-W2}
\cW_2 (t)=  \sum_{i=1}^k \int_0^1  (\lambda_i p_i)' (x)  |w_i(t, x)|^q \, dx -   \sum_{i=k+1}^m \int_0^1 ( \lambda_{i} p_{i})' (x) |w_{i}(t, x)|^q \, dx \\[6pt]
 - \sum_{i=1}^k \int_0^1  (\lambda_{m+i} p_{m+i})' (x)  \Big| w_{m+ i}(t, x) - M_{i} \Big(w_{k+ 1} \big(t, a_{k+1, m+i}(x) \big), \dots, w_{m+ i -1}\big(t,  a_{m+i - 1, m+i}(x) \big)\Big) \Big|^q \, dx.
 \end{multline}

On the other hand,  \eqref{thm1-pA1},   \eqref{thm1-pA2}, and \eqref{thm1-pA3} imply
\begin{equation}\label{thm1-p3}
(\lambda_i p_i)' = - q \Lambda p_i  \quad  \mbox{ for } 1 \le i \le k,
\end{equation}
\begin{equation}\label{thm1-p4}
(\lambda_i p_i)' = q \Lambda p_i  \quad  \mbox{ for } k+1 \le i \le k+m.
\end{equation}
Using \eqref{thm1-p3} and \eqref{thm1-p4}, we derive from \eqref{thm1-W2} that
\begin{equation}\label{thm1-W2-p1}
\cW_2(t) = - q \Lambda \cV(t).
\end{equation}

We have, by the Gaussian elimination process,
\begin{equation*}
\sum_{i=j}^{k}\Big| w_{m+ i}(t, 0) - M_{i} \Big(w_{k+ 1} (t, 0), \dots, w_{m+ i -1} (t,  0)\Big) \Big| \ge C\sum_{i=j}^{k}  |(B w_+)_i (t, 0)|.
\end{equation*}
for $j=k$, then $j=k-1$, \dots, and finally for $j=1$.  Using the fact
$$
\int_0^1 \lambda_{i_1}^{-1}(s) \, d s < \int_0^1 \lambda_{i_2}^{-1}(s) \, d s \mbox{ for } 1 \le i_1 < i_2 \le k,
$$
and, for $a_i \ge 0$ with $1 \le i \le j \le k$ and $1\le q < + \infty$,
$$
\Big( \sum_{i=1}^j a_i \Big)^q \le C^q \sum_{i=1}^j a_i^q,
$$
for some positive constant $C$ independent of $q$ and $a_i$,
we derive from \eqref{thm1-pA1} and \eqref{thm1-pA3} that, for large $\Gamma$ (the largeness of $\Gamma$ depends only on $B$, $k$, and $l$; it is in particular independent of $\Lambda$ and $q$),
\begin{multline*}
\sum_{i = 1}^{k} \lambda_{m+i} (0) p_{m +i} (0) \Big| w_{m+ i}(t, 0) - M_{i} \Big(w_{k+ 1} (t, 0), \dots, w_{m+ i -1} (t,  0)\Big) \Big|^q \\[6pt]    \ge  \sum_{i = 1}^k  \lambda_i (0) p_i (0) | (B w_+)_i(t, 0)|^q.
\end{multline*}
It follows from \eqref{thm1-W1} that
\begin{equation}\label{thm1-W1-p2}
\cW_1 (t) \le 0.
\end{equation}

Combining \eqref{thm1-dV}, \eqref{thm1-wW}, \eqref{thm1-W2-p1}, and \eqref{thm1-W1-p2} yields
\begin{equation*}
\frac{d}{dt} \cV\big(w(t, \cdot) \big) \le - q\Lambda \cV\big(w(t, \cdot) \big).
\end{equation*}
This implies
\begin{equation}\label{thm1-decay}
\cV\big(w(t, \cdot) \big) \le e^{- q \Lambda t}  \cV \big(w(0, \cdot) \big).
\end{equation}

Set
\begin{equation}\label{thm1-Aa}
A = \mathop{\sup_{1 \le i \le n}}_{x \in (0, 1)} p_{i}(x) \quad  \mbox{ and } \quad a = \mathop{\inf_{1 \le i \le n}}_{x \in (0, 1)}p_{i}(x),
\end{equation}
and define, for $v \in [L^2(0, 1)]^n$,
\begin{multline}\label{norm-V}
\| v \|_{\cV}^q = \int_0^1 \sum_{i=1}^m |v_i(x)|^q \, dx \\[6pt]
+ \int_0^1  \sum_{i =1}^{k} \Big| v_{m+ i}(x) - M_{i} \Big(v_{k+ 1} \big(a_{k+1, m+i}(x) \big), \dots, v_{m+ i -1}\big( a_{m+i - 1, m+i}(x) \big)\Big) \Big|^q  \, dx.
\end{multline}

Using \eqref{thm1-pA1}, \eqref{thm1-pA2}, \eqref{thm1-pA3},   and the definition of $T_{opt}$ \eqref{def-Top}, one can check that
\begin{equation}\label{thm1-A/a}
A/a \le  C^q e^{q \Lambda T_{opt}},
\end{equation}
for some positive constant $C$ depending only on $\Gamma$ and $\Sigma$.  It follows that
\begin{multline*}
\| w(t, \cdot) \|_{\cV}^q \mathop{\le}^{\eqref{thm1-Aa}, \eqref{norm-V}} \frac{1}{a} \cV\big(w(t, \cdot) \big)
 \mathop{\le}^{\eqref{thm1-decay}} \frac{1}{a} e^{-q \Lambda t} \cV\big(w(0, \cdot) \big) \\[6pt]
 \mathop{\le}^{\eqref{thm1-Aa}, \eqref{norm-V}} \frac{A}{a}  e^{-q \Lambda t}  \| w_0 \|_{\cV}^q \mathop{\le}^{\eqref{thm1-A/a}} C^q e^{q \Lambda (T_{opt} - t)} \| w_0 \|_{\cV}^q.
\end{multline*}
Since $\| v \|_{\cV} \sim \| v \|_{L^q(0,1)}$ for $v \in \big[ L^q(0, 1) \big]^n$ by \Cref{lemA1} in the appendix, assertion~\eqref{thm1-cl1} follows.

\medskip
It is clear that \eqref{thm1-cl2} is a consequence of \eqref{thm1-cl1} by taking $q \to + \infty$.
 \qed

\subsection{Proof of \Cref{thm1} for $m < k$} \label{sect-LN2} The proof of \Cref{thm1} for $m<k$ is similar to the one for $m \ge k$.  Indeed, one  has
\begin{equation}\label{thm1-pp1}
\cW_2(t) = - \Lambda \cV.
\end{equation}
We have, by the Gaussian elimination process, for $k+1 \le m + j \le m +k$,
\begin{equation*}
\mathop{\sum_{i}}_{m+ j \le m + i \le m+k}\Big| w_{m+ i}(t, 0) - M_{i} \Big(w_{k+ 1} (t, 0), \dots, w_{m+ i -1} (t,  0)\Big) \Big| \ge C \mathop{\sum_{i}}_{m+ j \le m + i \le m+k} |(B w_+)_i (t, 0)|.
\end{equation*}
and, for $1 \le j \le k-m$,
\begin{equation*}
\mathop{\sum_{i}}_{k+1 \le m + i \le m+k}\Big| w_{m+ i}(t, 0) - M_{i} \Big(w_{k+ 1} (t, 0), \dots, w_{m+ i -1} (t,  0)\Big) \Big|
\ge C  |(B w_+)_j (t, 0)|.
\end{equation*}
Using the fact
$$
\int_0^1 \lambda_{i_1}^{-1}(s) \, d s < \int_0^1 \lambda_{i_2}^{-1}(s) \, d s \mbox{ for } 1 \le i_1 < i_2 \le k,
$$
we derive from \eqref{thm1-pA1} and \eqref{thm1-pA3} that, for large $\Gamma$ (the largeness of $\Gamma$ depends only on $B$, $k$, and $l$; it is in particular independent of $\Lambda$ and $q$),
\begin{multline*}
\mathop{\sum_{i}}_{k+1 \le m + i \le m+k} \lambda_{m+i} (0) p_{m +i} (0) \Big| w_{m+ i}(t, 0) - M_{i} \Big(w_{k+ 1} (t, 0), \dots, w_{m+ i -1} (t,  0)\Big) \Big|^q \\[6pt]    \ge  \sum_{i = 1}^k  \lambda_i (0) p_i (0) | (B w_+)_i(t, 0)|^q.
\end{multline*}
One can then derive that
\begin{equation}\label{thm1-pp2}
\cW_1 (t) \le 0.
\end{equation}
Combining \eqref{thm1-pp1} and \eqref{thm1-pp2} yields
$$
\frac{d}{dt} \cV(t) \le - \Lambda \cV(t).
$$
The conclusion now follows as in the proof of \Cref{thm1} for $m \ge k$. The details are omitted. \qed

\section{On the nonlinear setting} \label{sect-NL}

The  following result  was established in \cite{CoronNg20}.

\begin{proposition} \label{pro1} Assume that $\nabla B(0)  \in {\mathcal B}$.  For any $T > T_{opt}$, there exist $\eps > 0$ and a time-independent feedback control for \eqref{Sys}, \eqref{bdry-w-0}, and \eqref{bdry-w-1}  such that if the compatibility conditions $($at $x =0$$)$  \eqref{compatibility-0} and \eqref{compatibility-1} below hold for $w(0, \cdot)$,
$$
\left(\| w(0, \cdot) \|_{C^1([0, 1])} < \eps\right)\Rightarrow \left(w(T, \cdot) =0\right).
$$
\end{proposition}

In what follows, we denote, for $x \in [0,1]$ and $y \in \mR^n$,
$$
\Sigma_-(x, y)  = \mbox{diag} \Big(- \lambda_1(x, y), \cdots, - \lambda_{k}(x, y) \Big) \mbox{ and } \Sigma_+(x, y)  = \mbox{diag} \Big( \lambda_{k+1}(x, y), \cdots,  \lambda_{n}(x, y) \Big).
$$
The compatibility conditions considered in \Cref{thm2} are:
\begin{equation}\label{compatibility-0}
w_-(0, 0) = \B  \big( w_+(0, 0) \big)
\end{equation}
and
\begin{equation}\label{compatibility-1}
\Sigma_- \big(0, w(0, 0)\big) \partial_x w_-(0, 0)=\nabla \B \big( w_+(0, 0) \big)  \Sigma_+ \big(0, w(0, 0)\big) \partial_x w_+(0, 0).
\end{equation}

We next describe the feedback given in the proof of \Cref{pro1} in \cite{CoronNg20}. Let $x_j$ be defined as
$$
\frac{d}{dt}x_j(t, s, \xi) =   \lambda_j \Big(x_j(t, s, \xi), w \big(t, x_j(t, s, \xi) \big) \Big) \quad \mbox{ and } \quad x_j(s, s, \xi) = \xi  \mbox{ for } 1 \le j \le k,
$$
and
$$
\frac{d}{dt}x_j(t, s, \xi) = -  \lambda_j \Big(x_j(t, s, \xi), w \big(t, x_j(t, s, \xi) \big) \Big) \quad \mbox{ and } \quad x_j(s, s, \xi) = \xi  \mbox{ for } k+1 \le j \le k + m.
$$
We do not precise at this stage the domain of the definition of $x_j$. Later, we only consider the flows in the regions where the solution $w$ is well-defined.

To arrange the compatibility of our controls, we also introduce auxiliary variables
satisfying autonomous dynamics.  Set $\delta = T- T_{opt} > 0$. For $t \ge 0$, let, for $k+1 \le j \le k+m$,
\begin{equation}
\zeta_{j}(0) = w_{0, j}(1), \quad \zeta_j'(0) = \lambda_j\big(1, w_0(1)\big) w_{0, j}'(1), \quad \zeta_{j}(t) = 0 \mbox{ for } t \ge \delta/2,
\end{equation}
and
\begin{equation}
\eta_j(0) = 1, \quad \eta_j'(0)=0, \quad \eta_{j}(t) = 0 \mbox{ for } t \ge \delta/2.
\end{equation}

We first deal with the case $m \ge k$.
Consider the last equation of \eqref{bdry-w-0}. Impose the condition $w_{k}(t, 0) = 0$. Using \eqref{cond-B-1} with $i = 1$ and the implicit function theorem, one can then write the last  equation of \eqref{bdry-w-0} under the form
\begin{equation}\label{kd-1-new}
w_{m+k}(t, 0) = M_k  \Big( w_{k +1 }(t, 0), \cdots, w_{m + k - 1} (t, 0) \Big),
\end{equation}
for some $C^2$ nonlinear map  $M_k$ from $U_k$ into $\mR$ for some neighborhood $U_k$ of $0 \in \mR^{m-1}$ with $M_k(0) = 0$ provided that $|w_{+}(t, 0)|$ is sufficiently small.

Consider the last two equations of \eqref{bdry-w-0} and impose the condition $w_{k}(t, 0) = w_{k-1}(t, 0) = 0$. Using \eqref{cond-B-1} with $i = 2$ and  the Gaussian elimination approach,  one can then write these two equations under the form \eqref{kd-1-new} and
\begin{equation}\label{kd-2-new}
w_{m+k-1}(t, 0) = M_{k-1} \Big( w_{k +1 }(t, 0), \cdots, w_{m + k - 2} (t, 0) \Big),
\end{equation}
for some $C^2$ nonlinear map  $M_{k-1}$ from $U_{k-1}$ into $\mR$ for some neighborhood $U_{k-1}$ of $0 \in \mR^{m-2}$ with $M_{k-1}(0) = 0$ provided that $|w_{+}(t, 0)|$ is sufficiently small, etc. Finally, consider the  $k$ equations of \eqref{bdry-w-0} and impose the condition $w_{k}(t, 0) = \dots = w_{1}(t, 0) = 0$. Using \eqref{cond-B-1} with $i = k$ and the Gaussian elimination approach,  one can then write these $k$ equations under the form \eqref{kd-1-new}, \eqref{kd-2-new}, \dots, and \begin{equation}\label{kd-k}
w_{m+1}(t, 0) = M_{1} \Big(w_{k +1 }(t, 0), \cdots, w_{m} (t, 0) \Big),
\end{equation}
for some $C^2$ nonlinear map  $M_1$ from $U_1$ into $\mR$ for some neighborhood $U_1$ of $0 \in \mR^{m-k}$ with $M_1(0) = 0$ provided that $|w_{+}(t, 0)|$ is sufficiently small for $m>k$. When $m=k$, we just define $M_1 = 0$.

\medskip
We are ready to construct a feedback law for the null-controllability at the time $T$.  Let $t_{m+k}$ be such that
$$
x_{m+k}(t+ t_{m+k}, t, 1) = 0.
$$
It is clear that $t_{m+k}$ depends only on the current state $w(t, \cdot)$. Let $D_{m+k} = D_{m+k}(t) \subset \mR^2$ be the open set whose boundary is $\{ t \} \times [0, 1]$, $[t, t+ t_{m+k}] \times \{0\}$, and $\Big\{ \big(s, x_{m+k}(s,  t, 1) \big); \ s \in [t, t+ t_{m+k}] \Big\}$. Then $D_{m+k}$ depends only on the current state as well. This implies
$$
x_{k+1}(t, t+ t_{m+k}, 0), \dots, x_{k+m-1}(t, t+ t_{m+k}, 0) \mbox{ are well-defined by the current state $w(t, \cdot)$}.
$$
As a consequence, the feedback
\begin{multline}\label{bdry-1-NL}
w_{m+ k}(t, 1) =  \zeta_{m+k}(t)  \\[6pt]+ ( 1 - \eta_{m+k}(t)) M_{k} \Big(w_{k+ 1}\big(t, x_{k+1} (t, t+ t_{m+k},  0) \big), \dots, w_{k+ m-1}\big(t, x_{k+ m - 1} (t, t+  t_{m + k }, 0) \big)\Big)
\end{multline}
is well-defined by the current state $w(t, \cdot)$.

We then consider the system \eqref{Sys}, \eqref{bdry-w-0}, and the feedback \eqref{bdry-1-NL}. Let $t_{m+k-1}$ be such that
$$
x_{m+k-1}(t+ t_{m+k-1}, t, 1) = 0.
$$
It is clear that $t_{m+k-1}$ depends only on the current state $w(t, \cdot)$ and the feedback law \eqref{bdry-1-NL}. Let $D_{m+k-1} = D_{m+k-1}(t) \subset \mR^2$ be the open set whose boundary is $\{ t \} \times [0, 1]$, $[t, t+ t_{m+k-1}] \times \{0\}$, and $\Big\{ \big(s, x_{m+k-1}(s, t, 1) \big); \ s \in [t, t+  t_{m+k-1}] \Big\}$. Then $D_{m+k-1}$  depends only on the current state. This implies
$$
x_{k+1}(t, t+ t_{m+k-1}, 0), \dots, x_{k+m-2}(t, t+  t_{m+k-1}, 0) \mbox{ are well-defined by the current state $w(t, \cdot)$}.
$$
As a consequence, the feedback
\begin{multline}\label{bdry-2-NL}
w_{m+ k - 1}(t, 1) = \zeta_{m+k-1}(t)  \\[6pt]
+  ( 1- \eta_{m+k-1}(t)) M_{k-1} \Big(w_{k+ 1}\big(t, x_{k+1} (t, t+ t_{m + k - 1}, 0) \big), \dots, w_{k+ m-2}\big(t, x_{k+ m - 2} (t, t+  t_{m + k -1}, 0) \big) \Big)
\end{multline}
is well-defined by the current state $w(t, \cdot)$.

We continue this process and reach the system \eqref{Sys}, \eqref{bdry-w-0}, \eqref{bdry-1-NL}, \dots
\begin{multline}\label{bdry-m-1-NL}
w_{m+ 2}(t, 1) = \zeta_{m+2}(t)  \\[6pt]
+  (1 -  \eta_{m+2}(t)) M_{2} \Big(w_{k+ 1}\big(t, x_{k+1} (t, t+  t_{m + 2}, 0) \big), \dots, w_{m+1}\big(t, x_{m+1} (t, t+  t_{m + 2}, 0) \big) \Big).
\end{multline}
Let $t_{m+1}$ be such that
$$
x_{m+1}(t+ t_{m+1}, t, 1) = 0.
$$
It is clear that $t_{m+1}$ depends only on the current state $w(t, \cdot)$ and the feedback law \eqref{bdry-1-NL}, \dots, \eqref{bdry-m-1-NL}. Let $D_{m+1} = D_{m+1}(t) \subset \mR^2$ be the open set whose boundary is $\{ t \} \times [0, 1]$, $[t, t+ t_{m+1}] \times \{0\}$, and $\Big\{ \big(s, x_{m+1}(s, t, 1) \big); \ s \in [t, t + t_{m+1}] \Big\}$. Then $D_{m+1}$ depends only on the current state. This implies
$$
x_{k+1}(t, t+ t_{m+1}, 0), \dots, x_{m}(t, t+ t_{m+1}, 0) \mbox{ are well-defined by the current state $w(t, \cdot)$}.
$$
As a consequence, the feedback
\begin{multline}\label{bdry-m-NL}
w_{m+ 1}(t, 1) =  \zeta_{m+1}(t) \\[6pt]
+ (1 -  \eta_{m+1}(t))  M_{1} \Big(w_{k+ 1} \big(t, x_{k+1} (t, t + t_{m + 1}, 0) \big), \dots, w_{m}\big(t, x_{m} (t, t+  t_{m +1}, 0) \big)\Big)
\end{multline}
is well-defined by the current state $w(t, \cdot)$.

To complete the feedback for the system, we consider, for $k+1 \le j \le m$,
\begin{equation}\label{bdry-k-m}
w_{j}(t, 1) = \zeta_j(t),
\end{equation}

We next consider the case $k > m$. The feedback law is then given as follows
\begin{multline*}
w_{m+ k}(t, 1) =  \zeta_{m+k}(t)  \\[6pt]+  (1 - \eta_{m+k}(t)) M_{k} \Big(w_{k+ 1}\big(t, x_{k+1} (t, t+ t_{m+k},  0) \big), \dots, w_{k+ m-1}\big(t, x_{k+ m - 1} (t, t+  t_{m + k }, 0) \big)\Big),
\end{multline*}
\dots
\begin{equation*}
w_{k+2}(t, 1) = \zeta_{k+2}(t)  \\[6pt]
+  ( 1-\eta_{k+2}(t))  M_{k+2-m} \Big(w_{k+ 1}\big(t, x_{k+1} (t, t+ t_{k+2}, 0) \big) \Big),
\end{equation*}
and
$$
w_{k+1}(t, 1) = \zeta_{k+1}(t) + (1 - \eta_{k+1}(t)) M_{k+1 - m},
$$
with the convention $M_{k+1-m} = 0$.

\begin{remark} \rm
The feedbacks above are {\it time-independent} and the well-posedness of the control system  is established in \cite[Lemma 2.2]{CoronNg20} for small initial data.
\end{remark}

To introduce the Lyapunov function,
as in the linear setting, for $k+1 \le i < j  \le k+m$, and for $x \in [0, 1]$,  $t \ge \delta/2$, let $\tau(j, t, x)$
be such that
\begin{equation*}
x_{j} \big(\tau(j, t, x), t, x \big) = 0,
\end{equation*}
and define
\begin{equation*}
a _{i, j} (t, x) = a _{i, j} \big(x, w(t, \cdot) \big)= x_{i} \big(t, \tau(j, t, x), 0 \big).
\end{equation*}
In the last identities, by convention, we considered $x_{i} \big(t, \tau(j, t, x), 0 \big)$ as  a function of $t$ and $x$ denoted by $a _{i, j} (t, x)$ or  a function of $x$ and $w(t, \cdot)$ denoted by $a _{i, j} \big(x, w(t, \cdot) \big)$.

Set
$$
\cH = \Big\{v \in [C^1([0, 1]) ]^n; \mbox{ $v$ satisfies the compatibility conditions at 0 and 1} \Big\}.
$$
Let $q \ge 1$ and let  $\cV: \cH  \to \mR$ $(q \ge 1)$ be defined by
\begin{equation}
\cV (v) =  \hcV (v) + \tcV (v).
\end{equation}
Here, with $\ell = \max\{m, k\}$,
\begin{multline}\label{def-tV-NL}
\hcV(v) =  \sum_{i=1}^{\ell}  \int_0^1  p_i(x)|v_i(x)|^q \, dx \\[6pt]
+ \mathop{\sum_{i}}_{\ell + 1 \le m + i \le k+m}  \int_0^1 p_{m +i}(x)  \Big| v_{m+ i}(x) - M_{i} \Big(v_{k+ 1} \big(a^v_{k+1, m+i}(x, v) \big), \dots, v_{m+ i -1}\big( a^v_{m +i - 1, m+i}(x, v) \big)\Big) \Big|^q  \, dx,
\end{multline}
and
\begin{multline}\label{def-hV-NL}
\tcV(v) =  \sum_{i=1}^{\ell}  \int_0^1  p_i(x)|\partial_t v(0, x)|^q \, dx
+ \mathop{\sum_{i}}_{\ell + 1 \le m + i \le k+m}  \int_0^1 p_{m +i}(x)  \Big|  \partial_t v_{m+i} (0, x) \\[6pt] - \partial_t \Big(M_{i} \Big(v_{k+ 1} \big(t, a_{k+1, m+i}^v(t, x) \big), \dots, v_{m+ i -1}\big(t,  a_{m +i - 1, m+i}^v(t, x) \big)\Big)  \Big)_{t =0}\Big|^q  \, dx.
\end{multline}
Here  $v(t, \cdot)$ is the corresponding solution with $v(t=0, \cdot) = v$ and $a_{k+j, m+i}^v$ is defined as $a_{k+j, m+i}$ with $w(t, \cdot)$ replaced by $v(t, \cdot)$. We also define here
\begin{equation}\label{thm3-pN1}
p_i(x) = \lambda_i^{-1}(x, 0) e^{ - q \Lambda \int_0^x \lambda_i^{-1}(s, 0) \, ds + q \Lambda \int_0^1 \lambda_i^{-1}(s, 0) \, ds} \quad \mbox{ for } 1 \le i \le k,
\end{equation}
\begin{equation}\label{thm3-pN2}
p_i(x) = \Gamma^q \lambda_i^{-1}(x, 0) e^{q \Lambda \int_0^x \lambda_i^{-1}(s, 0) \, ds} \quad \mbox{ for } k + 1 \le i \le \ell,
\end{equation}
\begin{equation}\label{thm3-pN3}
p_{m +i} (x) = \Gamma^q \lambda_{m+i}^{-1} (x, 0) e^{q \Lambda \int_0^x \lambda_{m+i}^{-1}(s, 0) \, ds + q \Lambda \int_0^1 \lambda_{i}^{-1}(s, 0) \, ds} \quad \mbox{ for } \ell + 1 \le m + i \le m + k,
\end{equation}
for some large positive constant $\Gamma \ge 1$ depending only on $\Sigma$ and $B$ (it is independent of $\Lambda$ and $q$).

\medskip

Concerning the feedback given above, we have

\begin{theorem}\label{thm2} Let $m, \, k \ge 1$.  There exists a constant $C \ge 1$, depending only on $\B$ and $\Sigma$ such that for $\Lambda \ge 1$ and for
$T > T_{opt}$, there exist $\eps > 0$  such that if the compatibility conditions $($at $x =0$$)$  \eqref{compatibility-0} and \eqref{compatibility-1} hold for $w(0, \cdot)$, and  $
\| w(0, \cdot) \|_{C^1([0, 1])} < \eps$, we have, for $t \ge \delta/2$ with $\delta = T - T_{opt}$,
\begin{equation}\label{thm2-cl1}
\| w(t, \cdot) \|_{W^{1, q} (0, 1)} \\[6pt]
\le C e^{\Lambda \big( T_{opt} - t\big)} \Big(  \| w(0, \cdot) \|_{W^{1, q}(0, 1)}  + \| \zeta\|_{C^1} + \| \eta\|_{C^1} \| w(0, \cdot) \|_{W^{1, q}(0, 1)}  \Big).
\end{equation}
As a consequence, we have
\begin{multline}\label{thm2-cl2}
\| w(t, \cdot) \|_{C^1 ([0, 1])} \\[6pt]
\le C e^{\Lambda \big( T_{opt} - t\big)} \Big(  \| w(0, \cdot) \|_{C^{1}([0, 1])}  + \| \zeta\|_{C^1} + \| \eta\|_{C^1} \| w(0, \cdot) \|_{C^{1}([0, 1])}  \Big).
\end{multline}

\end{theorem}

\begin{proof} We first claim that, for $k+1 \le i < j \le k+m$ and $x \in [0, 1]$,
\begin{equation}\label{claim-1-NL}
\lambda_{i}\Big(a_{i, j}(t, x), w \big( t, a_{i, j}(t, x) \big) \Big) + \partial_t a_{i, j}(t, x) =
\lambda_{j}\big(x,  w(t, x) \big) \partial_x a_{i, j}(t, x).
\end{equation}
Indeed, by the characteristic, we have
\begin{equation*}
a_{i, j} \big(s, x_j(s, t, x) \big)  =   x_i(s, \tau(j, t, x), 0) \mbox{ for } t \le s \le \tau(j, t,  x).
\end{equation*}
Taking the derivative with respect to $s$ yields, for $t \le s \le \tau(j, t,  x)$,
\begin{equation*}
\partial_t a_{i, j} \big(s, x_j(s, t, x) \big) + \partial_s x_j(s, t, x) \partial_x  a_{i, j} \big( s, x_j(s, t, x) \big)  =   \partial_s x_i(s, \tau (j, t, x), 0).  \end{equation*}
Considering $s = t$ and using the definition of the flows, we obtain the claim.

As a consequence of \eqref{claim-1-NL}, we have
\begin{equation}\label{da-NL}
\partial_x \Big( w_{i} \big(t, a_{i, j}(t, x) \big) \Big) =  \frac{ \lambda_{i}\Big(a_{i, j}(t, x), w \big( t, a_{i, j}(t, x) \big) \Big) + \partial_t a_{i, j}(t, x) }{\lambda_{j}\big(x,  w(t, x) \big)} \partial_x w_{i} \big(t, a_{i, j}(t, x) \big).
\end{equation}
Identity \eqref{da-NL} is a variant of \eqref{thm1-du} for the nonlinear setting and plays a role in our analysis.

\medskip
We next only consider the case $m \ge k$. The case $m < k$ can be proved similarly as in the proof of \Cref{thm1}.  
We will assume that the solutions are of class $C^2$. The general case can be established via a density argument as in \cite[page 1475]{2015-Bastin-Coron-SICON} and \cite[Comments 4.6, page 127-128]{2016-Bastin-Coron-book}.

We first deal with  $\hcV$.  We have, for $t \ge \delta/2$,
\begin{multline}\label{dV-NL}
\frac{d}{dt} \hcV(w(t, \cdot))  =  - \sum_{i=1}^k \int_0^1  p_i(x) \lambda_i \big(x, w(t, x) \big)  \partial_x |w_i(t, x)|^q  \, dx \\[6pt]  + \sum_{i=k+1}^m \int_0^1  p_i(x) \lambda_i \big(x,  w(t, x)\big)  \partial_x |w_i(t, x)|^q  \, dx \\[6pt]
+ \sum_{i =1}^{k} \int_0^1 q  p_{m +i}(x)  \partial_t  T_{m+i }(t, x) |T_{m+i }(t, x)|^{q-2} T_{m+i }(t, x) \, dx,
\end{multline}
where
\begin{equation}\label{T-NL}
T_{m+i}(t, x) = w_{m+ i}(t, x) - M_{i} \Big(w_{k+ 1} \big(t, a_{k+1, m+i}(t, x) \big), \dots, w_{m+ i -1}\big(t, a_{m+i - 1, m+i}(t, x) \big)\Big).
\end{equation}
Using \eqref{da-NL} and noting that, for $k+ 1 \le i \le j \le k+m$,
$$
\partial_t w_i(t, a_{i, j} (t, x)) = \lambda_i \big(a_{i, j} (t, x), w(t, a_{i, j} (t, x)) \big)  \partial_x w_i(t, a_{i, j} (t, x)),
$$
one can prove that
\begin{equation}\label{ST-NL}
\partial_t T_{m+i}(t, x) = \lambda_{m+i} (x, w(t, x))\partial_x T_{m+i} (t, x).
\end{equation}

Using \eqref{ST-NL} and making an integration by parts,  as in \eqref{thm1-wW}, we obtain
\begin{equation}\label{thm1-wW-NL}
\frac{d}{dt} \hcV( w(t, \cdot))) = \hcW_1(t) + \hcW_2 (t),
\end{equation}
where
 \begin{multline}\label{thm1-W1-NL}
 \hcW_1(t) = - \sum_{i = 1}^k  \lambda_i (1, w(t, 1)) p_i (1) |w_i(t, 1)|^q   +  \sum_{i = 1}^k  \lambda_i (0, w(t, 0)) p_i (0) | (Bu_+)_i(t, 0)|^q \\[6pt] - \sum_{i=k+1}^m \lambda_i (0, w(t, 0)) p_i(0)  |w_i(t, 0)|^q \\[6pt]
  -   \sum_{i = 1}^{k} \lambda_{m+i} (0, w(t, 0)) p_{m +i} (0) \Big| w_{m+ i}(t, 0) - M_{i} \Big(w_{k+ 1} (t, 0), \dots, w_{m+ i -1} (t,  0)\Big) \Big|^q,
\end{multline}
and
\begin{multline}\label{thm1-W2-NL}
\hcW_2 (t)=  \sum_{i=1}^k \int_0^1  \big(\lambda_i(x, w(t, x)) p_i(x)\big)_x  |w_i(t, x)|^q \, dx -   \sum_{i=k+1}^m \int_0^1 \big( \lambda_{i} (x, w(t, x)) p_{i} (x) \big)_x  |w_{i}(t, x)|^q \, dx   \\[6pt]
 - \sum_{i=1}^k \int_0^1  \big(\lambda_{m+i}(x, w(t, x)) p_{m+i} (x) \big)_x   \Big| w_{m+ i}(t, x)  \, dx  \quad \quad \quad \quad \quad  \quad \quad \quad \quad  \\[6pt]  \quad \quad \quad \quad \quad \quad  \quad \quad \quad \quad \quad \quad - M_{i} \Big(w_{k+ 1} \big(t, a_{k+1, m+i}(t, x) \big), \dots, w_{m+ i -1}\big(t,  a_{m+i - 1, m+i}(t, x) \big)\Big) \Big|^q \, dx.
 \end{multline}

As in the proof of \Cref{thm1}, we also have, for large $\Gamma$ and $| w(t, 0)|$ sufficiently small,
\begin{multline*}
\sum_{i = 1}^{k} \lambda_{m+i} (0, w(t, 0)) p_{m +i} (0) \Big| w_{m+ i}(t, 0) - M_{i} \Big(w_{k+ 1} (t, 0), \dots, w_{m+ i -1} (t,  0)\Big) \Big|^2 \\[6pt]
\ge  \sum_{i = 1}^k  \lambda_i (0, w(t, 0)) p_i (0) | (B w_+)_i(t, 0)|^2.
\end{multline*}
This implies
\begin{equation}\label{thm3-NL-W1}
\hcW_1(t) \le 0.
\end{equation}
Concerning $\hcW_2(t)$, we write
$$
\lambda_i(x, w(t, x)) p_i(x) = \frac{\lambda_i (x, w(t, x))}{\lambda_i(x, 0)} \lambda_i(x, 0) p_i(x).
$$
Note that, since $\Sigma$ and  $\partial_y \Sigma$ are of class $C^1$,
\begin{equation*}
\left| \frac{\lambda_i(x, w(t, x))}{\lambda_i(x, 0)} - 1 \right|  + \left| \partial_x \left( \frac{\lambda_i(x, w(t, x))}{\lambda_i(x, 0)} \right) \right| \le C (\eps, \delta), 
\end{equation*}
a quantity which goes to 0 if $\eps \to 0$ for fixed $\delta$. 

Using \eqref{thm3-pN1} and \eqref{thm3-pN3},
we obtain
\begin{equation}\label{thm3-NL-W2}
\hcW_2 (t) \le  - q \Lambda (1 - C( \eps, \delta)) \hcV(t).
\end{equation}

Combining \eqref{thm1-wW-NL}, \eqref{thm3-NL-W1}, and \eqref{thm3-NL-W2} yields
\begin{equation}\label{thm2-dt-p1}
\frac{d}{dt} \cV(t)  \le - q (\Lambda - C (\eps, \delta)) \cV(t) \mbox{ for } t \ge \delta/2.
\end{equation}

We next investigate $\tcV$. By \eqref{def-hV-NL},  we have, for $t \ge \delta/2$,
\begin{multline}
\tcV(w(t, x)) =  \sum_{i=1}^{k}  \int_0^1  p_i(x)|\partial_t w(t, x) |^q \, dx
+ \mathop{\sum_{i}}_{k + 1 \le m + i \le k+m}  \int_0^1 p_{m +i}(x)  \Big| \partial_t w_{m+i}(t, x) \\[6pt]
-  \Big( M_{i} \Big(w_{k+ 1} \big(t, a_{k+1, m+i}(t, x) \big), \dots, w_{m+ i -1}\big(t, a_{m+i - 1, m+i}(t, x) \big) \Big)_t
\Big|^q  \, dx.
\end{multline}
Using \eqref{ST-NL}, we have
\begin{align*}
\frac{d}{dt} \tcV(w(t, \cdot))
= &  - \sum_{i=1}^k \int_0^1  p_i(x) \lambda_i \big(x, w(t, x) \big)  \partial_x |\partial_t w_i(t, x)|^q \, dx \\[6pt]
&+   \sum_{i=1}^k \int_0^1 \frac{q  p_i(x)}{\lambda_i(x, w(t,x))} \partial_y \lambda_i (x, w(t, x)) \partial_t w (t, x) |\partial_t w_i (t, x) |^q \, dx   \\[6pt]  &+ \sum_{i=k+1}^m \int_0^1  p_i(x) \lambda_i \big(x,  w(t, x)\big)  \partial_x |\partial_t w_i(t, x)|^q  \, dx  \\[6pt]
& + \sum_{i=k+1}^m \int_0^1 \frac{q p_i(x)}{\lambda_{i}(x, w(t, x))} \partial_y \lambda_i \big(x,  w(t, x)\big)  \partial_t w (t, x) |\partial_t w_i(t, x)|^q  \, dx  \\[6pt]
& + \sum_{i =1}^{k} \int_0^1  p_{m +i}(x)  \lambda_{m+i}(x, w(t, x)) \partial_x (|\partial_t T_{m+i }(t, x)|^{q} ) \, dx \\[6pt]
& + \sum_{i =1}^{k} \int_0^1 \frac{q  p_{m +i}(x)}{\lambda_{m+i}(x, w(t, x))}  \partial_y \lambda_{m+i} (x, w(t, x)) \partial_t w(t, x)  |\partial_t T_{m+i }(t, x)|^{q}  \, dx.
\end{align*}
Set
\begin{multline}\label{thm1-tW3-NL}
\tcW_3(t) = \sum_{i=1}^k \int_0^1\frac{ q  p_i(x)}{\lambda_i(x, w(t, x))} \partial_y \lambda_i (x, w(t, x)) \partial_t w (t, x) |\partial_t w (t, x) |^q \, dx  \\[6pt]
+ \sum_{i=k+1}^m \int_0^1 \frac{q p_i(x)}{\lambda_i(x, w(t, x))} \partial_y \lambda_i \big(x,  w(t, x)\big)  \partial_t w(t, x) |\partial_t w(t, x)|^q  \, dx  \\[6pt]
+ \sum_{i =1}^{k} \int_0^1  \frac{q  p_{m +i}(x) }{\lambda_{m+i} (x, w(t, x))}\partial_y \lambda_{m+i} (x, w(t, x)) \partial_t w(t, x)  |\partial_t T_{m+i }(t, x)|^{q} \, dx.
\end{multline}
An integration by parts yields
\begin{equation}\label{thm1-twW-NL}
\frac{d}{dt} \tcV( w(t, \cdot))) = \tcW_1(t) + \tcW_2 (t) + \tcW_3(t),
\end{equation}
where
 \begin{multline}\label{thm1-tW1-NL}
 \tcW_1(t) = - \sum_{i = 1}^k  \lambda_i (1, w(t, 1)) p_i (1) |\partial_t w_i(t, 1)|^q   +  \sum_{i = 1}^k  \lambda_i (0, w(t, 0)) p_i (0) | \partial_t (Bu_+)_i(t, 0)|^q \\[6pt] - \sum_{i=k+1}^m \lambda_i (0, w(t, 0)) p_i(0)  |\partial_t w_i(t, 0)|^q \\[6pt]
  -   \sum_{i = 1}^{k} \lambda_{m+i} (0, w(t, 0)) p_{m +i} (0) \Big| \partial_t w_{m+ i}(t, 0) - \Big( M_{i} \Big(w_{k+ 1} (t, 0), \dots, w_{m+ i -1} (t,  0)\Big)  \Big)_t\Big|^q,
\end{multline}
and
\begin{multline}\label{thm1-tW2-NL}
\tcW_2 (t)=  \sum_{i=1}^k \int_0^1  \big(\lambda_i(x, w(t, x)) p_i(x)\big)_x  |\partial_t w_i(t, x)|^q \, dx \\[6pt] -   \sum_{i=k+1}^m \int_0^1 \big( \lambda_{i} (x, w(t, x)) p_{i} (x) \big)_x  |\partial_t w_{i}(t, x)|^q \, dx   \\[6pt]
 - \sum_{i=1}^k \int_0^1  \big(\lambda_{m+i}(x, w(t, x)) p_{m+i} (x) \big)_x   \Big| \partial_t w_{m+ i}(t, x)  \quad \quad \quad \quad \quad  \quad \quad \quad \quad  \\[6pt]  \quad \quad \quad \quad \quad \quad  \quad \quad \quad \quad \quad \quad - \Big( M_{i} \Big(w_{k+ 1} \big(t, a_{k+1, m+i}(t, x) \big), \dots, w_{m+ i -1}\big(t,  a_{m+i - 1, m+i}(t, x) \big)\Big)  \Big)_t \Big|^q \, dx.
\end{multline}
As before, we have
\begin{equation}\label{thm1-twW12}
\tcW_1(t) + \tcW_2 (t) \le - q \Lambda (1 - C (\eps, \delta)) \tcV.
\end{equation}
One can check that
\begin{equation}\label{thm1-twW3}
\tcW_3 \le C (\eps, \delta) q \tcV.
\end{equation}
From \eqref{thm1-twW-NL}, \eqref{thm1-twW12}, and \eqref{thm1-twW3}, we derive that
\begin{equation}\label{thm2-dt-p2}
\frac{d}{dt} \tcV (t) \le - q \Lambda (1 - C(\eps, \delta)) \tcV.
\end{equation}

Combining \eqref{thm2-dt-p1} and \eqref{thm2-dt-p2} yields
\begin{equation*}
\frac{d}{dt} \cV (t) \le - q \Lambda (1 - C (\eps, \delta)) \cV.
\end{equation*}

The conclusion now follows as in the linear case after taking $\eps$ sufficiently small, replacing $\Lambda (1 - C \eps) $ by $\Lambda$, and noting that
$$
\| w(t, \cdot) \|_{C^1([0, 1])}  \le C  \Big(  \| w(0, \cdot) \|_{C^1([0, 1])} + \|\zeta \|_{C^1} + \|\eta \|_{C^1} \| w(0, \cdot) \|_{C^1([0, 1])} \Big)  \mbox{ for } 0 \le t \le \delta/2.
$$
We also note here that the conclusion \eqref{A-equiv} of \Cref{lemA1} also holds for nonlinear maps $M_i$ of class $C^1$ with $M_i(0)  = 0$ provided that $\|v \|_{C^1([0, 1)}$ is sufficiently small. The details are omitted.
\end{proof}

\appendix

\section{A useful lemma} \label{A}
\renewcommand{\theequation}{A\arabic{equation}}
\renewcommand{\thelemma}{A\arabic{lemma}}
  \setcounter{equation}{0}  
  \setcounter{lemma}{0}  

\begin{lemma} \label{lemA1} Let $m, \,  k  \ge 1$.
For $k + 1 \le i < j \le k+ m$, let $b_{i, j}: [0, 1] \to [0, 1]$ be of class $C^1$ such that
\begin{equation}\label{A-b}
c_1 \le  |b_{i, j}' (x)| \le c_2 \mbox{ for } x \in (0, 1),
\end{equation}
for some positive constants $c_1$ and $c_2$.
Set $\ell = \max \{k, m \}$. Let, for $\ell + 1 \le m + i \le m+k$,  $M_i \in \mR^{1 \times (m+1 - k - i)}$.  Define, for $v \in [L^q(0, 1)]^n$,
\begin{multline}\label{norm-AV}
\vertiii{v}^q =  \sum_{i=1}^{\ell} \int_0^1  |v_i(x)|^q \, dx \\[6pt]
+ \mathop{\sum_{i}}_{\ell + 1 \le m + i \le k+m} \int_0^1    \Big| v_{m+ i}(x) - M_{i} \Big(v_{k+ 1} \big(b_{k+1, m+i}(x) \big), \dots, v_{m + i -1}\big( b_{m +i - 1,  m +i}(x) \big)\Big) \Big|^q  \, dx.
\end{multline}
We have
\begin{equation}\label{A-equiv}
\lambda^{-1} \| v\|_{L^q(0, 1)} \le \vertiii{v}  \le \lambda \| v\|_{L^q(0, 1)},
\end{equation}
for some $\lambda \ge 1$ depending only on $k$, $m$, $c_1$, and $c_2$, and $M_i$; it is independent of $q$.
\end{lemma}

\begin{proof} We only consider the case $m \ge k$. The other case can be proved similarly. It is clear that
\begin{equation}\label{coucoucou}
\vertiii{v} \le C \| v\|_{L^q(0, 1)}.
\end{equation}
On the other hand, using the inequality, for $\xi_1, \xi_2 \in \mR^d$ with $d \ge 1$,
$$
|\xi_1|^q + |\xi_2 - \xi_1|^q \ge C^{-q} (|\xi_1|^q + |\xi_2|^q),
$$
we have, for $1 \le i \le k$,
\begin{multline}\label{lemA1-p1}
 \int_0^1  \Big| v_{m+ i}(x) - M_{i} \Big(v_{k+ 1} \big(b_{k+1, m+i}(x) \big), \dots, v_{m+ i -1}\big( b_{m+i - 1, m+i}(x) \big)\Big) \Big|^q  \, dx \\[6pt]
 +  \sum_{k+1 \le j \le  m + i - 1}  \int_0^1   |v_{i} \big(b_{j, m+i} (x) \big)|^q \, dx \ge C^{-q} \int_0^1   |v_{m+i}(x)|^q \, dx.
\end{multline}
Using \eqref{A-b}, by a change of variables, we obtain, for $k+1 \le i < j \le m + k$,
\begin{equation}\label{lemA1-p2}
   \int_0^1 |v_{i} \big(b_{i, j} (x) \big)|^q \, dx \le C \int_0^1  |v_{i}(x)|^q \, dx.
\end{equation}
From \eqref{lemA1-p1} and \eqref{lemA1-p2}, we deduce that
\begin{multline}\label{lemA1-p3}
\sum_{i =1}^{k} \int_0^1   \Big| v_{m+ i}(x) - M_{i} \Big(v_{k+ 1} \big(b_{k+1, m+i}(x) \big), \dots, v_{m+ i -1}\big( b_{m+i - 1, m+i}(x) \big)\Big) \Big|^q  \, dx \\[6pt]
+ \sum_{i=k+1}^m  \int_0^1 |v_i(x)|^q \,d x  \ge C^{-q} \int_0^1 \sum_{i=k+1}^n |v_i(x)|^q \,d x.
\end{multline}
The conclusion then follows from \eqref{coucoucou} and  \eqref{lemA1-p3}.
\end{proof}

\medskip
\noindent \textbf{Acknowledgments.} The authors were partially supported by  ANR Finite4SoS ANR-15-CE23-0007. H.-M. Nguyen thanks Fondation des Sciences Math\'ematiques de Paris (FSMP) for the Chaire d'excellence which allows him to visit  Laboratoire Jacques Louis Lions  and Mines ParisTech. This work has been done during this visit.

\providecommand{\bysame}{\leavevmode\hbox to3em{\hrulefill}\thinspace}
\providecommand{\MR}{\relax\ifhmode\unskip\space\fi MR }
\providecommand{\MRhref}[2]{%
  \href{http://www.ams.org/mathscinet-getitem?mr=#1}{#2}
}
\providecommand{\href}[2]{#2}

\end{document}